\documentclass[11pt]{article} 
 \usepackage{amsthm,amssymb,amsmath}
\theoremstyle{plain}
\newtheorem{prop}{Proposition}[section]
\newtheorem{lemma}[prop]{Lemma}
\newtheorem{theorem}[prop]{Theorem}
\newtheorem{cor}[prop]{Corollary}
\newtheorem*{teoA}{Theorem A}

\theoremstyle{definition}
\newtheorem{defi}[prop]{Definition}

\newtheorem{remark}[prop]{Remark}
\newtheorem{example}[prop]{Example}

\newcommand{\R}{\mathbb{R}}       	%real
\newcommand{\C}{\mathbb{C}}       	%complex
\newcommand{\HH}{\mathbb{H}}	%quaternion

\newcommand{\M}{\mathcal{M}}

\newcommand{\norm}[1]{\vert #1 \vert}

\newcommand{\sdet}{\mathrm{Sdet}} %Study determinant

\newcommand{\car}{\mu}

\newcommand{\id}{{ I}}      %identitiy matrix

\newcommand{\diag}{\mathrm{diag}}

\newcommand{\tr}{\mathrm{tr}}

\newcommand{\lo}{\lambda_0}
\newcommand{\gmatrix}[1]{ \begin{bmatrix}#1\end{bmatrix} }
\newcommand{\pematrix}[1]{\mbox{\small $\gmatrix{#1}$}}
\newcommand{\mimatrix}[1]{\mbox{\footnotesize $\gmatrix{#1}$}}

\begin{document}

\title{Characteristic functions and Hamilton-Cayley theorem for left eigenvalues of quaternionic matrices\footnote{Partially supported by FEDER and MICINN Spain, Research Project MTM2008-05861}}

\author{E. Mac\'{\i}as-Virg\'os\footnote{\tt quique.macias@usc.es }\and M. J. Pereira-S\'aez\footnote{\tt mariajose.pereira@usc.es}}
 
\maketitle

%============================

\begin{abstract} We introduce the notion of characteristic function  of a quaternionic matrix, whose roots  are the left eigenvalues.   We prove that for all $2\times 2$ matrices and for $3\times 3$ matrices having some zero entry outside the diagonal there is a characteristic function  which is a polynomial. For the other $3\times 3$ matrices the characteristic function is a rational function with one point of discontinuity. We prove that  Hamilton-Cayley theorem holds in all cases.
\end{abstract}



\section{Introduction}
Very little is known about left eigenvalues of $n\times n$ quaternionic matrices.  F.~Zhang's papers \cite{ZHANG1997,ZHANG2007} review their main properties as well as some pathological examples, see also \cite{HUANG2000}. For $n=2$ the explicit computation of the left spectrum  is due to L.~Huang and W.~So \cite{HUANGSO}, while the authors studied the symplectic group in \cite{MP1,MP2}.

In 1985, R.~M.~W.~Wood \cite{WOOD1985} proved, by using homotopic methods, that every quaternionic matrix has at least one left eigenvalue. At the end of his paper,  Wood notes that ``in the $2 \times 2$ case  of the matrix
$\pematrix{%
a & b\cr
c&d\cr
}$
there is a partially defined determinant $b - ac^{-1}d$ and partially defined characteristic equation  
\begin{equation}\label{CHARWOOD}
\lambda c^{-1}\lambda-\lambda c^{-1}d-a c^{-1}\lambda-b+[ac^{-1}d] =0
\end{equation}
which reduces the eigenvalue problem to the fundamental theorem [of algebra]. The difficulties 
start with $3\times 3$ matrices''. 

In this paper we introduce a definition of characteristic function  for a quaternionic matrix, which generalizes the usual characteristic polynomial in the real and complex setting. In particular, its roots are the left eigenvalues. Explicitly, we say that $\car\colon \HH \to \HH$ is a characteristic function of the matrix $A\in\M(n,\HH)$ if, up to a constant, 
its norm verifies that $\norm{\car(\lambda)}=\sdet(A-\lambda \id )$ for all $\lambda\in\HH$, where $\sdet\colon \M(n,\HH)\to [0,+\infty)$ is Study's determinant. As we shall see, this definition fits naturally with Equation~(\ref{CHARWOOD}), as well as with the method proposed by W.~So in \cite{SO2005}   to compute the left eigenvalues when $n=3$.

Then we discuss Hamilton-Cayley theorem in this setting. Our main result is as follows.

\begin{teoA}\label{MAIN}
For any  quaternionic matrix $A\in\M(n,\HH)$, $n\leq 3$, there exists a characteristic function $\mu$   whose extension to a map $\mu\colon \M(n,\HH) \to \M(n,\HH)$ verifies Hamilton-Cayley, that is $\mu(A)=0$.
\end{teoA}

For $n=2$, a characteristic function like that in (\ref{CHARWOOD}) is a polynomial $\mu(\lambda)$ for which it is easy to check that $\car(A)=0$. It follows that
$$Ac^{-1}A=Ac^{-1}d+ac^{-1}A+(b-ac^{-1}d)\id ,$$ which  generalizes
the well known formula $A^2=(\tr A) A -(\det A) \id $ in the commutative setting.
When $n=3$ and the matrix has some zero entry outside the diagonal, we shall find a polynomial characteristic function that verifies Hamilton-Cayley. Otherwise, there is a characteristic function which is, outside a point of discontinuity, a rational function. We are able to extend it   to a map $\car\colon \M(n,\HH) \to \M(n,\HH)$ and we prove by brute force that Hamilton-Cayley is verified too.

At the end of the paper we discuss another possible definition of characteristic function.

\section{Preliminaries}\label{BACK}
We consider the quaternionic space $\HH^n$ as a {\em right} vector space over $\HH$. Two square matrices $A,B\in\M(n,\HH)$ are {\em similar}  if $B=PAP^{-1}$ for some invertible square matrix $P$. 

If $A$ is a quaternionic $n\times n$ matrix, let us write $A=X+jY$, with $X,Y\in \M(n,\C)$, and let
$$c(A)=
\gmatrix{%
X&-\overline{Y}\cr
Y & \overline{X}\cr
}\in \M(2n,\C)$$
be its {\em complex form}. We have $c(A\cdot B)=c(A)\cdot c(B)$, $c(A+B)=c(A)+c(B)$ and $c(tA)=tc(A)$ if $t\in\R$. In particular, $A$ is invertible if and only if $c(A)$ is invertible. Moreover, $\det c(A)\geq 0$ is a nonnegative real number, so we can
define the {\em Study's determinant} of $A$ as
\begin{equation}\label{STUDYDET}
\sdet(A)=(\det c(A))^{1/2}\geq 0.
\end{equation}
For complex matrices, $\sdet$ equals the norm of the complex determinant, see \cite{ASLAKSEN1996,COHENLEO2000} for a general discussion of quaternionic determinants.  The following properties are immediate:
\begin{enumerate}
\item
$\sdet(A\cdot B)=\sdet(A)\cdot\sdet(B)$;
\item
$A$ is invertible if and only if $\sdet(A)> 0$;
\item
if $A, B$ are similar matrices then $\sdet(A)=\sdet(B)$.
\end{enumerate}

We also need the following result.
\begin{lemma}\label{BOX}
 For a matrix with boxes $M,N$ of size $m\times m$ and $n\times n$ respectively we have 
$$\sdet
\gmatrix{%
0&M\cr
N&*\cr
}=\sdet(M)\cdot\sdet(N).
$$
\end{lemma}
It follows that $\sdet(A)= \norm{ q_1\cdots q_n}$ when   $A$ is a triangular matrix, with $q_1,\dots,q_n$ being the elements of the diagonal.

Sometimes we shall  permute two columns or rows of the matrix $A$. Or we shall add to a column a right linear combination of the columns. This will not affect the value of the determinant because the matrices of the type $P=
\pematrix{
0&1&0\cr
1&0&0\cr
0&0&1}$ or 
$P=\pematrix{1&0&0\cr
\alpha&1&0\cr
\beta&0&1\cr} $ verify $\sdet(P)=1$.

\begin{remark}
Up to the exponent $1/2$ in (\ref{STUDYDET}), this is the same determinant that the one in Theorem 8.1 of \cite{ZHANG1997} that we shall refer to later in Sect.~\ref{ULTIMA}. 
The exponent  is normalized in order to have
$\sdet(A)= \norm{ q_1\cdots q_n}$ for a diagonal matrix $A=\diag(q_1,\dots,q_n)$.
\end{remark}

\section{Left eigenvalues and characteristic functions}
A quaternion  $\lambda\in\HH$ is  said to be a {\em left  eigenvalue} of the matrix $A\in\M(n,\HH)$ if
$Av = \lambda v$    for some vector $v\in \HH^n$, $v\neq 0$. Equivalently,  the matrix $A-\lambda \id $ is not invertible, that is $\sdet(A-\lambda \id )=0$, where $\sdet$ is Study's determinant defined in Section \ref{BACK}.

\begin{defi} A map $\car \colon \HH \to \HH$ is a {\em characteristic function} of the matrix $A\in\M(n,\HH)$ if, up to a constant, $\norm{\car(\lambda)}=\sdet(A-\lambda \id )$ for all 
$\lambda\in\HH$.
\end{defi}

Notice that $\lambda$ is a left eigenvalue of $A$ if and only if $\mu(\lambda)=0$. 

\begin{remark}\label{PERMUT}
It is well known that the left spectrum is not invariant under similarity. However, if $P$ is a {\em real} invertible matrix then $\sdet(PAP^{-1}-\lambda \id )=\sdet(A-\lambda \id )$, so $A$ and $PAP^{-1}$ have the same characteristic functions. \end{remark}

\begin{example} Diagonal and triangular matrices.
\end{example}
If $A=\diag(q_1,\dots,q_n)$ then $\car(A)=(q_n-\lambda)\cdots(q_1-\lambda)$ is a characteristic function. Analogously for triangular matrices.

\begin{example}
 $2\times 2$ matrices.
 \end{example}
Let 
$A=
\gmatrix{%
a&b\cr
c&d\cr
}\in \M(2,\HH)$.
If $b= 0$ then
$\sdet(A)=\norm{da}$ and 
the map
$
\car(\lambda)=(d-\lambda)(a-\lambda)
$
 is a characteristic function. 
If $b\neq 0$ we have
$$A\sim 
\gmatrix{%
0 & b\cr
c-db^{-1}a & d
}$$
so
$$\sdet(A)=\norm{b}\norm{c-db^{-1}a}.$$

Consequently, we consider the characteristic function
\begin{equation}\label{CHARDOS}
\car(\lambda)=c-(d-\lambda)b^{-1}(a-\lambda).
\end{equation}

Obviously, the characteristic function of a matrix is not unique. For instance, by permuting rows and columns we can obtain
$\mu(\lambda)=b-(a-\lambda)c^{-1}(d-\lambda)$ which is Wood's function in Equation (\ref{CHARWOOD}) (there is a misprint in the original article). However, as we shall see in Section \ref{TRES}, it is preferable to take minors starting from the top right corner as we do.

\section{Characteristic function of $3\times 3$ matrices}\label{TRES}
Now let
$A=\pematrix{%
a&b&c\cr
f&g&h\cr
p&q&r\cr
}$
 be a $3\times 3$ quaternionic matrix.
The computation of $\sdet(A)$  can be done as follows (a similar algorithm is valid for any $n> 3)$.

\subsection{Case $n=3, c\neq 0$}\label{CNO0}
First we consider the generic case when $c\neq 0$.   In this case   we can create zeroes in the first row,
$$A\sim
\gmatrix{%
0&0&c\cr
f-hc^{-1}a & g-hc^{-1}b&h\cr
p-rc^{-1}a &q-rc^{-1}b&r\cr
}.
$$
By Lemma \ref{BOX} and the $2\times 2$ case, it follows:  
\begin{prop}\label{CASOC0}
 If $c\neq 0$, then $\sdet(A)$ is given:
\begin{enumerate}
\item 
when $g-hc^{-1}b \neq 0$, by
$$
\norm{c}\cdot\norm{g-hc^{-1}b}\cdot\norm{p-rc^{-1}a-(q-rc^{-1}b)(g-hc^{-1}b)^{-1}(f-hc^{-1}a)};
$$
\item
when $g-hc^{-1}b=0$,
by
$$
\norm{c}\cdot\norm{q-rc^{-1}b}\cdot\norm{f-hc^{-1}a}.
$$
\end{enumerate}
\end{prop}
\smallskip
%\begin{remark}

%Notice that its norm is (up to $\norm{ the Study's determinant of the submatrix
%$\gmatrix{%
%b&c\cr
%g&h\cr
%}.
%$
%\end{remark}

\begin{cor}\label{POLOCOR} Let us call $\lambda_0=g-hc^{-1}b$ the {\em pole} of $A$. 
 Then
\begin{equation}
\sdet(A-\lambda_0 \id )=
\norm{c}\cdot\norm{q-(r-\lambda_0)c^{-1}b}\cdot\norm{f-hc^{-1}(a-\lambda_0)}.
\end{equation}
\end{cor} 

By applying Prop.~\ref{CASOC0} and Cor.~\ref{POLOCOR} to $A-\lambda \id $  we find the following characteristic function of $A$. 

\begin{defi}\label{CHARTRES}
When $c\neq 0$, a characteristic function for the $3\times 3$ matrix $A$ can be defined as follows:
\begin{enumerate}
\item
if $\lambda_0=g-hc^{-1}b$ is the pole of $A$, 
$$\car(\lambda_0)=\left(q-(r-\lambda_0)c^{-1}b\right)\left(f-hc^{-1}(a-\lambda_0)\right);$$  
\item
otherwise,
\begin{eqnarray*}
\car(\lambda)&=&
(\lambda_0-\lambda)
\left[\left(p-(r-\lambda)c^{-1}(a-\lambda)\right)\right.-\cr
&&\left.\left(q-(r-\lambda)c^{-1}b\right)
	 (\lambda_0-\lambda)^{-1}
	\left(f-hc^{-1}(a-\lambda)\right)\right].
\end{eqnarray*}
\end{enumerate}
\end{defi}

\begin{remark} In \cite{SO2005}, W.~So proved that the left eigenvalues of a $3\times 3$ matrix can be computed as roots of certain polynomials of degree $\leq 3$. Even though our computation is different from his, we obtain that the function in Def.~\ref{CHARTRES}  is exactly So's formula in \cite[p.~563]{SO2005}. This is why we have chosen to compute determinants starting from the top right corner.
\end{remark}

\subsection{Case $n=3, c=0$}\label{SSC0}
We briefly review what happens when $c=0$. First, if both  $b,h=0$ we have a triangular matrix, then we can take
\begin{equation}\label{BH0}
\car(\lambda)=(r-\lambda)(g-\lambda)(a-\lambda).
\end{equation}  
If $b= 0$ but $h\neq 0$ we can reduce to the  $2\times 2$ case by Lemma \ref{BOX}, so we take
\begin{equation}\label{B0}
\car(\lambda)=\left(q-(r-\lambda)h^{-1}(g-\lambda)\right)(a-\lambda).
\end{equation}  
Finally, if $b\neq 0$ we can (see the proof of Theorem \ref{HCC0}) create a zero in the left top corner of $A-\lambda \id $ and then permute the second and last column, in order to reduce the matrix $(A-\lambda I)P$ to the $2\times 2$ case.
Alternatively, we can simply permute  the second and  last column and the second and last row  of $A$, in order to  obtain a matrix $PAP^{-1}$ with the same characteristic function, to which  Subsection \ref{CNO0} applies. Notice however that with the latter method we obtain a rational  function, not a polynomial.

\section{Continuity}
The following example shows that the characteristic function $\mu$ in Definition \ref{CHARTRES} may not be continuous, even if  its norm $\norm{\mu}$ is a continuous map.

Let
$$A=\gmatrix{%
0&i&1\cr
3i-k&0&1\cr
k&-1+j+k&0\cr
}.$$
Its pole (see Cor.~\ref{POLOCOR}) is $\lambda_0=-i$ and 
$$\car(\lambda_0)=(j+k)(2i-k)=1-i+2j-2k.$$ 

However, for $\lambda\neq\lambda_0$ we have
$$\car(\lambda)=(-i-\lambda)\left(
k-\lambda^2- (-1+j+k+\lambda i)(-i-\lambda)^{-1}(3i-k+\lambda)\right),$$
and by taking $\lambda=-i+\varepsilon j$, $\varepsilon\in\R$, with $\varepsilon\to 0$, we obtain
$$\lim_{\varepsilon\to 0}\car(-i+\varepsilon j)=1+i+2j+2k\neq \car(\lambda_0).$$

In fact, the limit  
$$\lim_{\varepsilon\to 0}\car(-i+\varepsilon q)= -q(j+k)q^{-1}(2i-k)$$ depends on $q$, so $\lim_{\lambda\to\lambda_0} \car(\lambda)$ does not exist.

It is an open question whether it is always possible to find a continuous characteristic function.

%Since $\car$ is not unique, whether it is equivalent to a polynomial or not remains an open question. So's results mean that this is true when $n=3$.

\section{Hamilton-Cayley theorem}
We now discuss Hamilton-Cayley theorem.
%Clearly it is true for triangular matrices.

\subsection{Case $n=2$}
\begin{theorem} Let $A=\pematrix{%
a & b\cr
c&d\cr
}$ be a $2\times 2$ quaternionic matrix. Let $\car(\lambda)=c-(d-\lambda)b^{-1}(a-\lambda)$ be the characteristic function defined in (\ref{CHARDOS}). Then $\car(A)=0$.
\end{theorem}

\proof
We have
$$\gmatrix{%
c&0\cr
0&c\cr
}-\gmatrix{%
d-a&-b\cr
-c&0\cr
}
\gmatrix{%
b^{-1}&0\cr
0&b^{-1}\cr
}
\gmatrix{%
0&-b\cr
-c&a-d\cr
}=\gmatrix{%
0&0\cr
0&0\cr
}.$$
\qed

%The well known formula $A^2=(\tr A) A -(\det A) I$ in the commutative setting generalizes as follows to the quaternionic case.
\begin{cor} $Ab^{-1}A=Ab^{-1}a+db^{-1}A+(c-db^{-1}a)\id .$
\end{cor}

\subsection{Case $n=3$, $c=0$}
For $n=3$, a direct computation will show that Hamilton-Cayley theorem is true when $c=0$ (see Section \ref{TRES}).

\begin{prop}\label{HCC0} Let $A=\pematrix{%
a&b&0\cr
f&g&h\cr
p&q&r\cr
}$. Let 
$\car(\lambda)$ be the characteristic function  defined in Subsection \ref{SSC0}. Then  $\car(A)=0$.\end{prop}

\proof
If $b,h=0$ we take formula (\ref{BH0}), so
$\car(A)$ equals
$$
\gmatrix{%
r-a&0&0\cr
-f&r-g&0\cr
-p&-q&0\cr
} \gmatrix{%
g-a&0&0\cr
-f&0&0\cr
-p&-q&g-r\cr
} 
\gmatrix{%
0&0&0\cr
-f&a-g&0\cr
-p&-q&a-r\cr
} =
\gmatrix{%
0&0&0\cr
0&0&0\cr
0&0&0\cr
}. 
$$

If $b=0$, $h\neq0$ we take formula (\ref{B0}), then we check
\begin{eqnarray*}
\gmatrix{%
r-a &0&0\cr
-f&r-g&-h\cr
-p&-q&0\cr
}
h^{-1}\gmatrix{%
g-a &0&0\cr
-f&0&-h\cr
-p&-q&(g-r)\cr
}
\gmatrix{%
0 &0&0\cr
-f&a-g&-h\cr
-p&-q&a-r\cr
}
&=&\\
q\gmatrix{%
0 &0&0\cr
-f&a-g&-h\cr
-p&-q&a-r\cr
},&&
\end{eqnarray*}
that is, $(r\id -A)h^{-1}(g\id -A)(a\id -A)=q(a\id -A)$, hence $\car(A)=0$.

If $b\neq 0$, we have
$$\sdet(A-\lambda \id)= \sdet \gmatrix{%
0 & 0 &b\cr
f-(g-\lambda)b^{-1}(a-\lambda) & h &g-\lambda \cr
p-qb^{-1}(a-\lambda) & r-\lambda & q \cr
},$$
so we are in the $2\times 2$ situation (see Lemma \ref{BOX}).
First, assume $h=0$ and let us take $\car(\lambda)=(r-\lambda)\left(f-(g-\lambda)b^{-1}(a-\lambda)\right)$. We check
\begin{eqnarray*}
\gmatrix{%
r-a&-b&0\cr
-f&r-g&0\cr
-p&-q&0\cr
}
\gmatrix{%
g-a&-b&0\cr
-f&0&0\cr
-p&-q&g-r\cr
}
b^{-1}
\gmatrix{%
0&-b&0\cr
-f&a-g&0\cr
-p&-q&a-r\cr
}
&=&\\
\gmatrix{%
r-a&-b&0\cr
-f&r-g&0\cr
-p&-q&0\cr
}f,&&
\end{eqnarray*}
that is
$(r\id -A)(g\id -A)b^{-1}(a\id -A)=(r\id -A)f$, hence $\car(A)=0$.

On the other hand, if $h\neq 0$ we take 
\begin{equation}\label{BNOHNO}
\car(\lambda)=p-qb^{-1}(a-\lambda)-(r-\lambda)h^{-1}\left(
f-(g-\lambda)b^{-1}(a-\lambda)\right).
\end{equation}
Then we compute
\begin{eqnarray*}
p\id -qb^{-1}(a \id -A)-(r\id -A)h^{-1}f&=&\\
\gmatrix{%
p-(r-a)h^{-1}f& q+bh^{-1}f&0\cr
qb^{1}f+fh^{-1}f&p-qb^{-1}(a-g)-(r-g)h^{-1}f&qb^{-1}h+f\cr
qb^{-1}p+ph^{-1}f&qb^{-1}q+qh^{-1}f&p-qb^{-1}(a-r)\cr
}&&\\
\end{eqnarray*}
and we check it equals
\begin{eqnarray*}
-\gmatrix{%
r-a&-b&0\cr
-f&r-g&-h\cr
-p&-q&0\cr
}
h^{-1}
\gmatrix{%
g-a&-b&0\cr
-f&0&-h\cr
-p&-q&g-r\cr
}
b^{-1}
\gmatrix{%
0&-b&0\cr
-f&a-g&-h\cr
-p&-q&a-r\cr
}&=&\\
-(r\id -A)h^{-1}(g\id -A)b^{-1}(a\id -A),&&
\end{eqnarray*}
hence $\car(A)=0$.
\qed

\begin{lemma} Let $A$ be a quaternionic matrix such that   $\mu(\lambda)=0$ for some quaternionic polynomial  $\mu(\lambda)$. Let $B=PAP^{-1}$ be a similar matrix, with $P$ a real matrix. Then $\mu(B)=0$.
\end{lemma}

\begin{proof}
Let $\mu(\lambda)=q_1\lambda q_2\lambda\cdots q_k\lambda q_{k+1}$ be a monomial. Then $\mu(B)=P\mu(A)P^{-1}$.
\end{proof}
Notice that the same result is true when $\mu(\lambda)$ is a rational function.\\

By permuting rows and columns (see Remark \ref{PERMUT}) we deduce:
\begin{cor} Let $A$ be a $3\times 3$ quaternionic matrix with some zero entry outside the diagonal. Then there exists a polynomial characteristic function $\car$ such that $\car(A)=0$.
\end{cor}

\begin{example}
\label{ESTASI} Let us consider the matrix
$A=\pematrix{%
1&i&i\cr
i&j&k\cr
0&-1&j\cr
}$. It is real similar to 
$\pematrix{%
j&-1&0\cr
k&j&i\cr
i&i&1\cr
}$,
whose characteristic function is given by  formula~(\ref{BNOHNO}), that is
$$\car(\lambda)=i+i(j-\lambda)+(1-\lambda)i\left(k+(j-\lambda)^2\right).$$ Then the following equation holds:
$$-AiA^2+AiAj+AkA+iA^2-iAj+A(i+j)-(i+k)A+(k-j)\id =0.$$
\end{example}
\subsection{Case $n=3,c\neq 0$}
When $c\neq 0$,   the characteristic function of the matrix $A$   is a rational function with a pole. We shall extend it to a map in the space of matrices in the following natural way.

Let $\lo=g-hc^{-1}b$ be the pole of $A$.  Let
$$f_0=f-hc^{-1}(a-\lo),$$
$$q_0=q-(r-\lo)c^{-1}b.$$

\begin{lemma}\label{TEMPRANO} The matrix $\lo \id-A$ is invertible if and only if $f_0,q_0\neq 0.$
\end{lemma}

\begin{proof}
By Corollary \ref{POLOCOR}, 
$\sdet(\lo\id-A)=\norm{c}\norm{q_0f_0}$.
\end{proof}
\begin{defi}\label{EXT}
 We define $\mu\colon \M(n,\HH) \to \M(n,\HH)$ as follows (see Definition \ref{CHARTRES} ):  
\begin{enumerate}
\item
if $ \lo\id -B $ is invertible, then $\mu(B)=  q_0f_0\id $;
\item
otherwise,
\begin{eqnarray*}
\mu(B)&=& (\lo\id-B)
\left[\left(p\id-(r-B)c^{-1}(a\id-B)\right)\right. -\\
&&\left.\left(q\id-(r\id-B)c^{-1}b\right)
	 (\lo-B)^{-1}
	\left(f\id-hc^{-1}(a\id-B)\right)\right].\\
\end{eqnarray*}
\end{enumerate}
\end{defi}

The following Proposition completes the proof of Theorem~A.
\begin{prop}\label{QUEBIEN}
 The map $\mu$ in Def. \ref{EXT} satisfies Hamilton-Cayley theorem, that is $\mu(A)=0$
\end{prop}

\begin{proof}
If $\lo\id -A$ is not invertible, then $\mu(A)=q_0f_0\id =0$ by Lemma \ref{TEMPRANO}. Otherwise
 it suffices to prove that
\begin{equation}\label{IZQ}
p\id-(r\id-A)c^{-1}(a\id-A)
\end{equation}
 equals
\begin{equation}\label{DER}
\left(q\id-(r\id-A)c^{-1}b\right) (\lo\id-A)^{-1}
\left(f\id-hc^{-1} (a\id-A)\right).
\end{equation}
%\end{proof}
\begin{lemma}\label{TERM1}
A direct computation shows that the first term (\ref{IZQ}) is
$$
{\mimatrix{%
-bc^{-1}f & -q+(r-a)c^{-1}b+bc^{-1}(a-g) & -bc^{-1}h\\
+(r-g)c^{-1}f+hc^{-1}p & p-fc^{-1}b-hc^{-1}q-(r-g)c^{-1}(a-g)
& -f+(r-g)c^{-1}h+hc^{-1}(a-r)\\
-qc^{-1}f & -pc^{-1}b+qc^{-1}(a-g) & -qc^{-1}h\\
}}.
$$
\end{lemma}

We now want to compute the term (\ref{DER}). 

We start by computing $(\lo\id-A)^{-1}$ by Gaussian elimination.

Let $$P_1=
\gmatrix{%
1&0&0\cr
0&1&0\cr
c^{-1}(\lo-a)&0&1\cr
}$$
and
$$P_2=
\gmatrix{%
1&0&0\cr
0&1&0\cr
0&-c^{-1}b&1\cr
}$$
Then
\begin{equation}\label{QUEHORAES}
(\lambda_0 I-A)P_1P_2=
\gmatrix{%
0&0&-c\cr
-f_0&0&-h\cr
-p^*&-q_0&\lo-r\cr
},
\end{equation}
where 
$$p^*=p-(\lo-r)c^{-1}(\lo-a).$$
The inverse of the matrix $(\lo\id-A)P_1P_2$ in (\ref{QUEHORAES}) can be computed by hand; it is
$$B= 
\gmatrix{%
f_0^{-1}hc^{-1}&-f_0^{-1}&0\cr
-q_0^{-1}n^*&q_0^{-1}p^*f_0^{-1}&-q_0^{-1} \cr
-c^{-1}&0&0\cr
}$$
where 
$$n^*= p^*f_0^{-1}hc^{-1}-(\lo-r)c^{-1}.$$
It follows that
$$(\lo I -A)^{-1}= P_1P_2B=$$
\begin{eqnarray*}\label{OZONO}
{\pematrix{
f_0^{-1}hc^{-1}&-f_0^{-1}&0\\
-q_0^{-1}n^*&q_0^{-1}p^*f_0^{-1}&-q_0^{-1} \\
c^{-1}(\lo-a)f_0^{-1}hc^{-1}+c^{-1}bq_0^{-1}n^*-c^{-1}&
-c^{-1}(\lo-a)f_0^{-1}-c^{-1}bq_0^{-1}p^*f_0^{-1}&+c^{-1}bq_0^{-1}\\
}}.\nonumber
\end{eqnarray*}

Moreover
\begin{eqnarray*}
F=f \id-hc^{-1}(a\id-A)&=&
\gmatrix{
f& hc^{-1}b & h\cr
hc^{-1}f &f-hc^{-1}(a-g) & hc^{-1}h\cr
hc^{-1}p & hc^{-1}q & f-hc^{-1}(a-r)\cr
},
\end{eqnarray*}
while
\begin{eqnarray*}
Q=q\id-(r\id-A)c^{-1}b=&=&
\gmatrix{
q-(r-a)c^{-1}b & bc^{-1}b & b \cr
fc^{-1}b & q-(r-g)c^{-1}b & hc^{-1}b \cr
pc^{-1}b & qc^{-1}b &q\cr
 }.
\end{eqnarray*}

We have to compute (\ref{DER}), that is $QP_1P_2BF$.

First we compute $(P_1P_2B)F$. For instance, its first column is given by
$$
[(P_1P_2B)F]^1=
\pematrix{
0\\
-q_0^{-1}n^*f+q_0^{-1}p^*f_0^{-1}hc^{-1}f-q_0^{-1}hc^{-1}p &\\
+c^{-1}bq_0^{-1}n^*f-c^{-1}f-c^{-1}bq_0^{-1}p^*f_0^{-1}hc^{-1}f+c^{-1}bq_0^{-1}hc^{-1}p}.
$$
%$$
%[(P_1P_2B)F]^2=
%\pematrix{
%f_0^{-1}hc^{-1}hc^{-1}b-f_0^{-1}f+f_0^{-1}hc^{-1}(a-g)\\
%-q_0^{-1}n^*hc^{-1}b+q_0^{-1}p^*f_0^{-1}f-q_0^{-1}p^*f_0^{-1}hc^{-1}(a-g)-q_0^{-1}hc^{-1}q &\\
%X},
%$$
%and
%$$
%[(P_1P_2B)F]^3=
%\pematrix{
%0\\
%-q_0^{-1}n^*h+q_0^{-1}p^*f_0^{-1}hc^{-1}h+c^{-1}bq_0^{-1}f-c^{-1}bq_0^{-1}hc^{-1}(a-r)\\
%X},
%$$
%\end{proof}

Now we check for instance the entry $(1,1)$ of the matrix
$Q(P_1P_2BF)$. We have

\begin{eqnarray*}
[Q(P_1P_2BF)]_1^1&=&\\
bc^{-1}b\left(-q_0^{-1}n^*f+q_0^{-1}p^*f_0^{-1}hc^{-1}f-q_0^{-1}hc^{-1}p\right)&+&\\
b\left(+c^{-1}bq_0^{-1}n^*f-c^{-1}f-c^{-1}bq_0^{-1}p^*f_0^{-1}hc^{-1}f+c^{-1}bq_0^{-1}hc^{-1}p\right)&=&\\
-bc^{-1}f\\
\end{eqnarray*}
which indeed is the entry  $(1,1)$ in Corollary \ref{TERM1}.

The other entries are computed in a similar way.
\end{proof}

\begin{example}
Let $A=\left(
\begin{array}{lll}
 1 & i & -j \\
 i & -1 & k \\
 1 & -1 & j
\end{array}
\right)$. The pole is $\lo=-2$ and $\mu(\lo)=-5+8j$. For $\lambda\neq\-2$, the characteristic function is
$$\mu(\lambda)=
-(2+\lambda)\left(2 +\lambda(-1+j) - \lambda j \lambda+(-1 + i-\lambda k)(2+\lambda)^{-1}i(2 - \lambda) \right).
$$
With the notations of the of proof of Proposition \ref{QUEBIEN},
it is
$$(\lo\id -A)^{-1} =
(1/12)\left(
\begin{array}{lll}
 -3 & 3 i & 0 \\
 2 i-j-k & -8+2 i+j+3 k & 2+2 i+4 k \\
 1+i-j & -3-i+2 j+k & 2 -2+j+k
\end{array}
\right).$$
$$P= \left(
\begin{array}{lll}
 -j & 1-i+3 k & 1 \\
 -k & 3-i-3 j & -i \\
 -k & -2 j+k & i
 \end{array}
 \right),
$$

$$
Q=
\left(
\begin{array}{lll}
 -1+i-k & j & i \\
 j & -1+i+k & 1 \\
 -k & k & -1
\end{array}
\right)$$

and

$$
F=\left(
\begin{array}{lll}
 i & 1 & k \\
 1 & 3 i & j \\
 -i & i & 2 i-k
\end{array}
\right).
$$
We have $P-Q(\lo\id-A)^{-1}F=0$.
\end{example}

\section{Final remarks}\label{ULTIMA}
In this Section we discuss a different approach to the definition of characteristic functions for left eigenvalues. 

In order to clarify concepts, let us briefly comment the same problem but for {\em right} eigenvalues. Let $c(A)\in\M(2n,\C)$ be the complex form  of the matrix $A\in\M(n,\HH)$ (see Section \ref{BACK}). Then, as it is well known, the right eigenvalues of $A$ are the quaternions  $qz q^{-1}$, where $q\in\HH$, $q\neq 0$, and $z$ is a complex eigenvalue of $c(A)$. It follows:
\begin{theorem}[\cite{ZHANG1997}] Let $p(z)=\det(c(A)- z \id )=\sum_{k=0}^{2n}{c_kz^k}$, $c_k\in\R$, be the  characteristic polynomial of $c(A)$. Then $p(A)=\sum_{k=0}^{2n}{c_k A^k}=0$.
\end{theorem}

Now, let $\lambda=x+jy$, with $x,y\in\C$, be a left eigenvalue of $A$. Equivalently, the matrix $c(A-\lambda \id )$ is not invertible. It follows that the left eigenvalues are the roots of the function $\sigma\colon \C\times \C \to \R$ given by
\begin{equation}\label{FINAL}
\sigma(x,y)=\det
\gmatrix{%
X-x\id  & -\overline{Y}+\overline{y}\id \cr
Y-y\id  & \overline{X}-\overline{x}\id \cr
}.
\end{equation}
Let $A=X+jY$, with $X,Y\in\M(n,\C)$. Then Hamilton-Cayley theorem could be stated as $\sigma(X,Y)=0$, provided this has a meaning. However we have the following counterexample even for $n=2$.

\begin{example} 
Let $A=\pematrix{%
0&i\cr
j&0\cr
}$. 
Let $x=x_1+ix_2$, $y=y_1+iy_2$. Then
$$\sigma(x,y)=1+(x_1^2+x_2^2+y_1^2+y_2^2)^2-4x_2y_1.$$
On the other hand, it is
$X=\pematrix{%
0&i\cr
0&0\cr
}$ and 
$Y=\pematrix{%
0&0\cr
1&0\cr
}$, so
$X_1=0$, $X_2=\pematrix{%
0&1\cr
0&0\cr
}$, $Y_1=\pematrix{%
0&0\cr
1&0\cr
}$ and $Y_2=0$, then    
$\sigma(X,Y)=
\pematrix{%
-3&0\cr
0&1\cr
}\neq 0$.
\end{example}

\noindent Institute of Mathematics.
Department of Geometry and Topology.\\
University of Santiago de Compostela.
15782- SPAIN

\end{document}